\documentclass[12pt]{amsart}
 \pdfoutput=1
\usepackage{cooltooltips}
\usepackage[all,arc,cmtip]{xy}
\usepackage{enumerate}
\usepackage{mathrsfs}
\usepackage{euscript}
\usepackage{hyperref}
\usepackage{amsthm}
\usepackage{amsmath}
\usepackage{MnSymbol,bbding,pifont}
\usepackage{tikz-cd}
\usepackage{xspace}
\usepackage{mathtools}
\usepackage[belowskip=-15pt,aboveskip=0pt]{caption}

   \usepackage{bm}

\usepackage[latin1]{inputenc}%................... Recognizes ê, ë, etc
\usepackage[T4]{fontenc}%........................ Type 4 fonts
\usepackage{textcomp}%........................... Additional text character

\usepackage{usnomencl}%.......................... List of symbols (in usthesis pack)

\usepackage{graphicx}%........................... Graphicx loaded in usthesis
\usepackage{color}%.............................. Color setup
\usepackage{eso-pic}%............................ Shipout commands for watermark

\usepackage{mathtools}
\usepackage{multicol}
\usepackage{paralist}
\usepackage{geometry}
\usepackage{algorithmic}
\usepackage[ruled,vlined]{algorithm2e}
\usepackage{comment}
\usepackage{tikz-cd}
\usepackage{url}
\usepackage{graphicx,scalerel}
\usetikzlibrary{decorations.pathmorphing}
\usepackage{float}
\usepackage{tabulary}
\usepackage{booktabs}
\newtheorem{theorem}{Theorem}[section]
\newtheorem{corollary}[theorem]{Corollary}
\newtheorem{deflem}[theorem]{Definition-Lemma}
\newtheorem{lemma}[theorem]{Lemma}
\newtheorem{proposition}[theorem]{Proposition}
\theoremstyle{definition}
\newtheorem{definition}[theorem]{Definition}
\newtheorem{remark}[theorem]{Remark}
\newtheorem{example}[theorem]{Example}

\newcommand{\suchthat}{\;\ifnum\currentgrouptype=16 \middle\fi|\;}

\newcommand{\af}[4]{\tenq{\mathcal{A}}_{\scalebox{0.5}{#1}\!} (#2,#3)_{\scalebox{0.6}{$#4$}}}
\newcommand{\afz}[3]{\tenq{\mathcal{A}}_{\scalebox{0.5}{#1}\!} (#2,#3)}

\newcommand\restr[2]{{% we make the whole thing an ordinary symbol
\left.\kern-\nulldelimiterspace % automatically resize the bar with \right
#1 % the function
\vphantom{\big|} % pretend it's a little taller at normal size
\right|_{#2} % this is the delimiter
}}

\newcommand{\sbfrac}[3]{\frac{ #1}{#2}{}^{\scaleto{ #3}{3.2 pt} }}
\makeatletter
\newcommand*{\doublerightarrow}[2]{\mathrel{
		\settowidth{\@tempdima}{$\scriptstyle#1$}
		\settowidth{\@tempdimb}{$\scriptstyle#2$}
		\ifdim\@tempdimb>\@tempdima \@tempdima=\@tempdimb\fi
		\mathop{\vcenter{
				\offinterlineskip\ialign{\hbox to\dimexpr\@tempdima+1em{##}\cr
					\rightarrowfill\cr\noalign{\kern.5ex}
					\rightarrowfill\cr}}}\limits^{\!#1}_{\!#2}}}

\usepackage{stackengine}
\stackMath
\newcommand\tenq[2][1]{%
 \def\useanchorwidth{T}%
  \ifnum#1>1%
    \stackunder[0pt]{\tenq[\numexpr#1-1\relax]{#2}}{\scriptscriptstyle\sim}%
  \else%
    \stackunder[1pt]{#2}{\scriptscriptstyle\sim}%
  \fi%
}

\title{\bfseries A note on the moduli spaces of free algebras of rank 2}
\author{Sophie Marques}

\begin{document}

\setcounter{tocdepth}{3}
\maketitle
\begin{center} 
Department of Mathematical Sciences, 
University of Stellenbosch, 
Stellenbosch, 7600,
South Africa\\
NITheCS (National Institute for Theoretical and Computational Sciences), 
South Africa \\
 e-mail: smarques@sun.ac.za
\end{center}

\begin{abstract}
In this paper, we present a formulation of the moduli problem for rank-2 algebras over general base rings in functorial terms, providing presentations as presheaf quotients of affine schemes by group scheme actions.
\\


%\noindent \textbf{2020 Math. Subject Class:}
%
%\noindent \begin{center}
%\rm e-mail: \href{mailto:d smarques@sun.ac.za}{ smarques@sun.ac.za}
%
%\it
%Department of Mathematical Sciences, 
%University of Stellenbosch, \\
%Stellenbosch, 7600, 
%South Africa\\
%\&
%NITheCS (National Institute for Theoretical and Computational Sciences), \\
%South Africa \\ \bigskip
%
%\rm e-mail: \href{mailto:d.mgani99@gmail.com}{d.mgani99@gmail.com}
%
%\it
%Department of Mathematical Sciences, 
%University of Stellenbosch, \\
%Stellenbosch, 7600, 
%South Africa
%
%
%\end{center} 

%\textbf{e-mail:} \ \ \ \href{mailto:d.mgani99@gmail.com}{d.mgani99@gmail.com} and \href{mailto:d smarques@sun.ac.za}{ smarques@sun.ac.za}
\end{abstract}
  \tableofcontents
  
  \section*{Acknowledgement}
The author gratefully acknowledges Qing Liu for numerous insightful discussions that have significantly shaped the development of this work especially the results of section \S2. The author also thanks John Voight for valuable conversations and helpful feedback.
\section*{Introduction}

The classification of finite-rank algebras up to isomorphism is a classical and extensively studied problem in algebra and number theory (see, for instance, \cite{Delone, Meyer, Kitamura, micali, Szeto, bhargava, bhargava2, poonen, voight, marques, mpendulo}). Such classification efforts aim to describe the set of isomorphism classes of algebras over a fixed base ring by identifying invariants and constructing parameter spaces that encode these classes in a coherent and geometrically meaningful way. These parameter spaces, often realized as {moduli spaces}, provide a systematic framework for organizing algebraic structures and offer powerful tools for investigating their geometry, arithmetic, and deformation theory.

This paper focuses on the moduli theory of free algebras of rank $2$ over commutative rings. These algebras occupy a particularly tractable corner of the broader landscape: they are necessarily quadratic, and in the separable case, they are always Galois. Rank-2 free algebras form the last and only family of finite-rank algebras that remains amenable to a complete and explicit classification. Nonetheless, even in this relatively simple setting, several structural subtleties remain unresolved, particularly concerning the geometry of their moduli spaces and the extent to which these spaces reflect deeper algebraic properties.

While the literature on this subject is extensive-including foundational contributions by Voight \cite{voight}, Hahn \cite{hahn}, Small \cite{small}, and Pirashvili \cite{Pirashvili}, it can be challenging to distinguish which results are standard and which remain underexplored. Where possible, we have cited appropriate references. 

We formulate the moduli problem for rank-2 algebras over a base ring~$R$ in functorial terms. Although such functors are not generally representable by schemes, we show that in several natural cases, including separable, (separable) radical, and Artin-Schreier algebras, the moduli functor admits a presentation as a presheaf quotient of an affine scheme by a group scheme action. The central goal of this paper is to make this statement precise and explicit

Our main result (Theorem~\ref{main}) establishes that the moduli functors associated with any of these families of algebras can be expressed as quotients of affine lines by group scheme actions, with respect to all the parameters we have identified. Although these presentations are not canonical, they nonetheless capture the inherent complexity of the moduli problems. Whether a canonical description exists, possibly one that minimizes the number of parameters, remains an open question. Theorem~\ref{main} is closely related to \cite[Theorems A, B, C, and Section 3, particularly pp. 499-500]{voight}, where some of the moduli spaces and their parametrizations were described with great detail. The primary focus of \cite{voight} was the uniqueness of the monoid structure they possess. We include these results here to provide a cohesive narrative that ties together the ideas we aim to develop.

A motivation behind this work is to examine the extent to which the structure of a moduli space encodes the properties of the objects it classifies. This perspective gives rise to several guiding questions, some of which remain open by the end of the paper, but nonetheless shape its overall organization:

\begin{itemize}
  \item \textsf{Dimension.} What is the minimal number of parameters required to classify rank-2 algebras over a given base? In Section~1, we explore this idea: while general quadratic algebras over an arbitrary base ring typically require two parameters, radical and Artin--Schreier extensions can be described by a single parameter, as are separable algebras via their discriminants. However an open question remains, can we formalize a notion of {minimal parametrization} on the moduli space directly? 

\item \textsf{Structure of the group action.}  A natural question is whether the moduli spaces of rank-2 free algebras satisfying certain algebraic properties can always be realized as quotients of affine schemes by group scheme actions. Theorem \ref{main} provides a positive answer for the cases of separable, Artin-Schreier, and radical algebras, though it is plausible that additional properties may also conform to this pattern. Another natural line of inquiry concerns the types of quotients and group actions that can arise as moduli spaces. Section 2 begins by analyzing what ultimately turn out to be the inertia groups associated with the action introduced in the final section, evaluated at maximal ideals. These groups also correspond to the automorphism groups of a given algebra. We further examine their fibers over algebraically closed fields. Notably, all three group schemes of order two, $\mathbb{Z}/2\mathbb{Z}$, $\mu_2$, and $\alpha_2$, appear in this setting, along with the infinite group scheme $\mathbb{G}_m$. Moreover, under suitable regularity assumptions on the base ring $R$, we show that the isomorphism classes are uniquely determined by the discriminant modulo squares (see Corollary \ref{regular}). We are also grateful to Jean Gillibert for pointing out a related interesting result in \cite[Proposition 2.26]{Dallaporta}.

  \item \textsf{Geometric and categorical directions.} One could further study the moduli problem by stackifying the groupoid associated with the presheaf quotient and analyzing the resulting quotient stack. Alternatively, one could adopt a categorical viewpoint to investigate what the ``correct'' setting is for understanding such moduli problems and whether deeper categorical structures are at play.
\end{itemize}

The structure of the paper is as follows. In Section~1, we develop the general functorial formalism for studying moduli of free algebras of rank 2, including localizations, points, and fiber products of the different functors we want to study. While some of this material may be familiar, it plays an essential role in subsequent arguments and is included for completeness.

Section~2 presents explicit parameterizations for the various moduli problems under consideration: free algebras of rank $2$, separable algebras, (separable) radical algebras, and Artin-Schreier algebras. For each case, we identify different types of parameters to describe those spaces. %and a corresponding group scheme that captures isomorphism classes via its action.

In Section~~3, we study the scheme that defines isomorphisms between algebras of rank~~2, as well as the group schemes governing automorphisms of a free algebra of the same rank. We analyze the fibers of these schemes and identify conditions under which the discriminant is sufficient to distinguish non-isomorphic algebras (see Corollary~\ref{propd}, Example~\ref{counter}, and Corollary~\ref{regular}).

In Section~4, we explicitly compute the Hopf algebra coactions corresponding to each group scheme action under consideration. While these derivations may be routine for specialists, we present them in full to bridge the algebraic and geometric perspectives. We proceed step by step: first computing the group structure and its action at the level of rings, then translating this via the Yoneda lemma and the functor of points into the category of schemes, and finally returning to rings using the equivalence between affine schemes and commutative rings. The coactions obtained are of intrinsic algebraic interest, as they enable the computation of inertia groups, higher ramification groups, local group actions at primes, the structure of the augmentation ideal, and other invariants that distinguish them from arbitrary coactions. We conclude by expressing each moduli space under consideration as a presheaf quotient of a affine scheme by a group scheme action.

Our goal here is to formulate the moduli problem for rank-2 algebras over a base ring $R$ in functorial terms, with a focus on explicit and trackable constructions. This approach aims to broaden the understanding of moduli spaces, an area still rich with open questions and lacking definitive answers. We hope this presentation serves as a valuable complement to existing work for researchers in algebraic geometry, algebra, and related fields.

%Our aim for future work is to gain a deeper understanding of the properties of these functorial moduli spaces, including their dimension and topological characteristics. Additionally, we seek to obtain explicit results for locally free algebras of rank 2, potentially drawing on the works of \cite{voight} and \cite{poonen}. We also hope to achieve a similar level of generality in studying ring extensions defined by cubic polynomials, extending the classification efforts presented in \cite{marques}.

\section{Notations and preliminaries} 
In this paper, $S$ denotes a commutative and unitary ring. 
We introduce the following categories:
\begin{itemize}
\item $S$-$\mathbf{Alg}$ denotes the category of $S$-algebras.
\item $\mathbf{Sets}$ denotes the category of sets.
\item $\mathbf{Fields}$ denotes the category of fields.
\end{itemize}
In this paper, "ring" refers to commutative and unitary rings, "algebra" refers to commutative and unitary algebras, and "morphisms" are unitary ring morphisms. Moreover, when \( D \) is a ring and \( T \) is a \( D \)-algebra with the structural map \( \phi: D \rightarrow T \),
\begin{itemize}
\item we denote \( \phi(u) \) simply as \( u \) for any \( u \in D \), unless it becomes necessary to distinguish them.
\item when \( T = D/I \) with \( I \) being an ideal of \( D \) and $\phi$ is the quotient map, we also denote \( \phi(u) \) as \( \overline{u} \). 
\item we denote the unit of $D$ as $1_D$.
\end{itemize}
Naturally, we adopt this notation when it does not obscure the intended meaning in context.
%We say that a {\sf free algebra of rank $2$} is inseparable over $R$ if the characteristic of $R$ is $2$ and $A \simeq \frac{R[x]}{x^2-a}$ for some $a\in R$.

We define radical and Artin-Schreier algebras, as they will be central to the focus of this paper.

\begin{definition}
\begin{itemize}
    \item An \( R \)-algebra of rank $2$ is called {\sf radical over $R$} if it is isomorphic to \( R[x] / \langle x^2 + a\rangle \) for some \( a \in R \). When a $R$-algebra $A$ is radical of rank $2$, we say that $u$ is a {\sf radical generator for $A$ over $R$} if $u^2\in R$ and $A=R \cdot 1_A \oplus R \cdot u$.
    \item An \( R \)-algebra of rank $2$ is called {\sf Artin-Schreier over $R$} if it is isomorphic to \( R[x] / \langle x^2 - x + a \rangle \) for some \( a \in R \).  When a $R$-algebra $A$ is Artin-Schreier, we say that $u$ is an {\sf Artin-Schreier generator for $A$ over $R$} if $u^2-u \in R$  and $A=R \cdot 1_A \oplus R \cdot u$.
\item An \( R \)-algebra \( A \) is called {\sf separable over \( R \)} if the multiplication map \(\mu: A \otimes_R A \to A\), defined by \(\mu(a \otimes b) = ab\) for any $a,b \in A$, admits a section \(\sigma: A \to A \otimes_R A\).

\item An \( R \)-algebra \( A \) is called {\sf \( G \)-Galois over \( R \)} if \( R = A^G \) (the ring of \( G \)-invariant elements in \( A \)) and the Galois map \(\gamma: A \otimes_R A \to \operatorname{Map}(G, A)\) is defined by \(\gamma(a \otimes b)(g) = a g(b)\) for all \( a, b \in A \) and \( g \in G \).
\end{itemize}
\end{definition}

\noindent To simplify notation in the sections below, we will use the following symbols:
\begin{itemize}
    \item "{\sf F}" for free.
    \item "{\sf SF}" for separable free.
    \item "{\sf R}" for radical.
    \item "{\sf SR}" for separable radical.
    \item "{\sf AS}" for Artin-Schreier.
        \item "{\sf SAS}" for separable Artin-Schreier.
\end{itemize}

Now for ${\sf P} \in \{ {\sf F}, {\sf SF}, {\sf R}, {\sf SR}, {\sf AS}, {\sf SAS}\}$.
We introduce the following notation: for any $S$-algebra $R$ and $R$-algebra $A$, and prime ideal $\mathfrak{p}$ of $S$,
\begin{itemize} 
\item \([A]_{{\sf P}, n, R}\) (or simply \([A]\) when there is no risk of confusion) denote the set of all algebras \({\sf P}\) of rank $n$ isomorphic to \(A\) as \(R\)-algebras.
\item $A_{\mathfrak{p}}$ denotes $A \otimes_S S_{\mathfrak{p}}$. 
\end{itemize} 
We now define the moduli space in a functorial manner as follows. These definitions can, of course, be generalized to any rank \( n \).

\begin{definition} Let ${\sf P} \in \{ \text{ {\sf F}, {\sf SF}, {\sf R}, {\sf SR}, {\sf AS} }\}$.
We define the {\sf (functorial) moduli space of the \( {\sf P} \) \( S \)-algebras of rank $n$ up to isomorphism} as the covariant functor \(\af{{\sf P}}{n}{-}{S}\) from the category of \( S \)-algebras to the category of sets defined as follows:

\begin{itemize}
    \item \(\af{{\sf P}}{n}{R}{S} = \{ [A] : A \text{ is an } R\text{-algebra}  \ {\sf P} \text{ of rank $n$} \}\), for any $S$-algebra $R$.
    \item \(\af{{\sf P}}{n}{\varphi}{S}\) is a morphism from \(\af{{\sf P}}{n}{R}{S}\) to \(\af{{\sf P}}{n}{R'}{S}\) that sends the set \( [A] \) to \( [A \otimes_R R'] \), for any morphism \(\varphi : R \rightarrow R'\) in \( S \)-\(\mathbf{Alg}\).
\end{itemize}

We denote \(\af{{\sf P}}{n}{-}{\mathbb{Z}}\) by \(\afz{{\sf P}}{n}{-}\). \\
Given a ring morphism $\psi: S\rightarrow S'$, we have a natural transformation $\iota_{\psi}$ from $\af{{\sf P}}{n}{-}{S}$ to $\af{{\sf P}}{n}{-}{S'}$, defined for all $S$-algebra $R$ as the map ${\iota_{\psi}}_R$ from $\af{{\sf P}}{n}{R}{S}$ to $\af{{\sf P}}{n}{R}{S'}$ sending $[A]_{{\sf P}, n, R}$ to $[A\otimes_S S']_{{\sf P}, n , R}$
\end{definition}

We can define a localized functorial moduli space as follows. 
\begin{definition} 
Let \({\sf P} \in \{ \text{{\sf F}, {\sf SF}, {\sf R}, {\sf SR}, {\sf AS}} \}\) and let \(\mathfrak{p}\) be a prime ideal in \(S\). We define the {\sf (functorial) moduli space of the \( S \)-algebras \( {\sf P} \) of rank $n$ localized at \(\mathfrak{p}\)} as the covariant functor \(\af{{\sf P}}{n}{R}{\mathfrak{p}}\) that maps from the category of \( S \)-algebras to the category of sets, defined as follows:

\begin{itemize} 
    \item \(\af{{\sf P}}{n}{R}{\mathfrak{p}} \) \(= \{ [A_{\mathfrak{p}}]_{{\sf P}, n, R_{\mathfrak{p}}} : A \text{ is an } R\text{-algebra } {\sf P} \}\), for any \(R\) \(S\)-algebra.
    \item \(\af{{\sf P}}{n}{\varphi}{\mathfrak{p}}\) is a morphism from \(\af{{\sf P}}{n}{R}{\mathfrak{p}}\) to \(\af{{\sf P}}{n}{R'}{\mathfrak{p}}\), which sends the set \([A_{\mathfrak{p}}]_{{\sf P}, n, R_{\mathfrak{p}}} \) to \( [(A \otimes_R R')_{\mathfrak{p}}]_{{\sf P}, n, R'_{\mathfrak{p}}} \), for any morphism \(\varphi : R \rightarrow R'\) in \(S\)-\(\mathbf{Alg}\).
\end{itemize}

\noindent We also have a natural transformation \(\ell_{\mathfrak{p}}\) from \(\af{{\sf P}}{n}{-}{S}\) to \(\af{{\sf P}}{n}{-}{\mathfrak{p}}\), defined for all \(S\)-algebras \(R\) as \({\ell_{\mathfrak{p}}}_R\) is the map from \(\afz{{\sf P}}{n}{R}\) to \(\af{{\sf P}}{n}{R}{\mathfrak{p}}\), sending \( [A] \) to \( [A_{\mathfrak{p}}]_{{\sf P}, n, R_{\mathfrak{p}}} \).
\end{definition}

%\begin{itemize}
%\item Given a ring $S$ and $A$ an $S$-algebra, we define $\mathcal{P}(A)_S$ to be the functor from $\mathbf{Rings}$ to $\mathbf{Sets}$, sending $R$ to $\{ A\otimes_S R \}$ and $\varphi: R\rightarrow R'$ to $\{\operatorname{id}_A \otimes \varphi\}$. 
%\item We define the functor $\mathcal{P}(-)_S$ from $S$-$\mathbf{Alg}$ to $\mathbf{Func}$, sending an $S$-$\mathbf{Alg}$ to $\mathcal{P}(A)_S$ and $\phi: A\rightarrow A'$ to $\mathcal{P}(\phi)_S: \mathcal{P}(A)_S \rightarrow \mathcal{P}(A')_S$, such that ${\mathcal{P}(\phi)_S}_R = \phi \otimes R$ for any $S$-algebra $R$. 
%\item Given a ring $S$ and $A$ a locally free $S$-algebra of rank $n$, we have a natural transformation $\xi_A : \mathcal{P}(A)_S \rightarrow \af{P}{n}{-}{S}$, which simply sends $A\otimes_S R $ to $\cf{A\otimes_R R'}{R}$, for any $S$-algebra. 
%\end{itemize}

%The constructions above are motivated by the following lemma:
\noindent We conclude this section with straightforward yet useful lemmas. In this paper, for any natural number $n$, $n \cdot 1_R$ denotes $1_R+ \cdots + 1_R$ n times and $(-n) \cdot 1_R= - (n \cdot 1_R)$.

\begin{lemma}\label{charp}
Let \( p \) be a prime number. The following statements are equivalent:
\begin{enumerate}
    \item there exist a prime ideal $\mathfrak{p}$ such that the residue field at $\mathfrak{p}$ is of characteristic $p$;
    \item \( p \) is not a unit in \( R \).
\end{enumerate}
Moreover, if \( k(\mathfrak{p}) \) is the residue field at a prime ideal $\mathfrak{p}$ of $R$ of characteristic $p$, then $p \in \mathfrak{p}$.
\end{lemma}

\begin{proof}
Let \(\varphi: \mathbb{Z} \rightarrow R\) be the only ring homomorphism from $\mathbb{Z}$ to $R$, and let \(\zeta_\mathfrak{p}: R \rightarrow k(\mathfrak{p})\) be the canonical morphism from \( R \) to the residue field at \( \mathfrak{p} \). It is not hard to prove that the following assertions are equivalent:
\begin{enumerate}
    \item \( k(\mathfrak{p}) \) has characteristic \( p \);
    \item \(\operatorname{ker} (\zeta_\mathfrak{p} \circ \varphi) = p\mathbb{Z}\);
    \item \(\zeta_\mathfrak{p} (p\cdot R) = 0\);
    \item \( p\cdot 1_R \in \mathfrak{p} \).
\end{enumerate}
The result follows directly from these equivalences, as being a unit is equivalent to not being contained in any prime ideal.
\end{proof}

\begin{lemma}\label{pq}
Let \( p \) and \( q \) be distinct prime numbers in $\mathbb{Z}$. We have \( q\cdot 1_R \in (p\cdot 1_R) R \) if and only if \( p\cdot 1_R \) is invertible in $R$.
\end{lemma}

\begin{proof}
We suppose that \( q\cdot 1_R \in (p\cdot 1_R) R \). Thus, there exists an element \( u \in R \) such that \( q \cdot 1_R = (p\cdot 1_R )u \). 

By B\'ezout's lemma, we can find integers \( a \) and \( b \) such that \( a p\cdot 1_R + bq\cdot 1_R= 1_A \). Substituting \( q\cdot 1_R = (p\cdot 1_R)u \) into this equation gives:

\[
a(p\cdot 1_R) + b((p\cdot 1_R)u) = 1,
\]

which simplifies to:

\[
(a + ub) \cdot (p\cdot 1_R) = 1.
\]
Therefore, \( p\cdot 1_R \) is invertible in $R$. The converse is straightforward, thus completing the proof.
\end{proof}

\section{Characterizing certain types of free algebras of rank $2$}
\noindent We observe that any free algebra of rank 2 can be defined by a single polynomial, a characteristic that does not typically extend to higher ranks. This property facilitates the more explicit exploration of the moduli space of free algebras of rank 2 in this paper. The proof is omitted and can be found in \cite[Lemma 3.2]{voight}.
\begin{proposition}  \label{free}
Let \( A \) be an \( R \)-algebra that is free of rank \( 2 \). Then there exist elements \( a, b \in R \) such that 
\[
\displaystyle A \simeq \frac{R[x]}{\langle x^2 + ax + b \rangle}.
\] 
More precisely, given a basis \( b_1, b_2 \) for \( A \) over \( R \), there exist elements \( u_1, u_2, r_1, r_2 \in R \) such that \( 1_A = u_1 \cdot   b_1 + u_2  \cdot  b_2 \) and \( r_1 u_1 + r_2 u_2 = 1_R \), and by defining \( e = -r_1 \cdot  b_1 + r_2  \cdot  b_2 \), we can express \( A \) as \( A = R \cdot 1_A \oplus R \cdot e \), where \( e^2 = -b \cdot 1_A - a \cdot e \) for some \( a, b \in R \). Mapping \( e \) to \( \bar{x} \) gives an isomorphism between \( A \) and \( \frac{R[x]}{\langle x^2 + ax + b \rangle} \).
\end{proposition}

One significant class of free algebras of rank \(2\) that emerges from the study of polynomials over a field with characteristic not equal to \(2\) is the class of radical algebras of rank \(2\). In the following proposition, we characterize the conditions under which a free algebra of rank \(2\) is radical, solely in terms of the coefficients of its defining polynomial, irrespective of the characteristic.

\begin{proposition}\label{radical}
Let 
\[
\displaystyle A = \frac{R[x]}{\langle x^2 + ax + b\rangle},
\]
where \(a, b \in R\). The following statements are equivalent: 
\begin{enumerate}
    \item \(A\) is a radical algebra.
    \item There exists \(v \in R\) such that \(a = 2v\).
\end{enumerate}
If $(2)$ is satisfied, \(\overline{x} + v\) is a radical generator for \(A\) over $R$ with the vanishing polynomial \(X^2 + b - v^2\).
In particular, if \(2\) is a unit in \(R\), then every free algebra of rank \(2\) over \(R\) is radical. Furthermore, any \(R\)-free algebra of rank \(2\), where \(R\) is a \(\mathbb{Z}_{\langle p\rangle }\)-algebra and \(p\) is a prime number not equal to \(2\), is also radical.
\end{proposition}

\begin{proof} 
The implication \((2) \Rightarrow (1)\) can be obtained by completing the square. To prove the converse, suppose that \(A\) is a radical algebra. Then, we can choose \(u \in \mathbb{R}^\times\) and \(v \in R\) such that \(u \cdot \overline{x} + v\) has a vanishing polynomial of the form \(x^2 + d\), for some \(d \in R\). 

Since \((u \overline{x} + v)^2 + d = 0\), we can expand this to obtain:
\[
\overline{x}^2 + 2u^{-1}v \overline{x} + u^{-2}(v^2 + d) = 0.
\]
Since $\{ 1, \overline{x}\}$ is a basis for $A$ over $R$, we deduce that \(a = 2u^{-1}v\). Thus, we conclude the proof. %To prove the last assertion, we can use Lemma \ref{pq}.
\end{proof}

In characteristic \(2\), field theory reveals that alongside radical extensions, Artin-Schreier extensions form a significant family of quadratic extensions. In the following proposition, we establish the conditions under which a free algebra of rank \(2\) is Artin-Schreier, solely in terms of the coefficients of the polynomial defining this algebra.

\begin{proposition}\label{AS}
Let $$ \displaystyle A = \frac{R[x]}{\langle x^2 + ax + b\rangle},$$ where \(a, b \in R\). The following assertions are equivalent:
\begin{enumerate}
    \item \(A\) is an Artin-Schreier algebra.
    \item There exists \(v \in R\) such that \(2v + a \in R^\times\).
\end{enumerate}
If $(2)$ is satisfied, setting $u:=-( 2v+a)$, \(u^{-1} (\overline{x} - v)\) is an Artin-Schreier generator for \(A\) over \(R\), with the vanishing polynomial given by \(X^2 - X+ u^{-2} (b - v^2 - uv)\).
\end{proposition}

\begin{proof}
(1) \(\Rightarrow\) (2): Suppose \(A\) is an Artin-Schreier algebra. Then, we can choose \(u \in R^\times\) and \(v \in R\) such that \(u \overline{x} + v\) has a minimal polynomial of the form \(x^2 - x + d\) for some \(d \in R\).

Since \((u \overline{x} + v)^2 - (u \overline{x} + v) + d = 0\), expanding this expression and multiplying the equality by $u^{-2}$ yields:
\[
\overline{x}^2 + \left( 2vu^{-1} - u^{-1} \right) \overline{x} + u^{-2}(v^2 -v + d) = 0.
\]
Since $\{ 1, \overline{x}\}$ is a basis for $A$ over $R$, we deduce that \(a = 2vu^{-1} - u^{-1}\). Therefore, $(2)$ is proven.

(2) \(\Rightarrow\) (1): We suppose that we are given \(v \in R\) and \(u \in R^\times\) such that \(2v + a = -u\), we define \(y = u^{-1} (\overline{x} - v)\). Then,
\[
y^2 - y = \left(u^{-1} (\overline{x} - v)\right)^2 - u^{-1} (\overline{x} - v).
\]
Expanding this gives:
\[
y^2 - y = u^{-2} (\overline{x}^2 - 2v\overline{x} + v^2) - u^{-1} (\overline{x} - v) = u^{-2} \left(\overline{x}^2 - (2v + u)\overline{x} + v^2 + uv\right).
\]
Substituting \(a = 2v + u\), we get:
\[
y^2 - y = u^{-2} \left(\overline{x}^2 + a\overline{x} + v^2 + uv\right) = u^{-2} (-b + v^2 + uv) \in R.
\]
Thus, we conclude the proof.
\end{proof}

The existence of an $R$-automorphism that acts as an involution on any locally free $R$-algebra of rank $2$ is well known and noted, for instance, in the introduction of \cite{Voight}. For completeness, we provide a global formula for this involution along with a proof.

\begin{proposition} \label{galois} Let $A$ be a locally free $R$-algebra of rank $2$. We define the map $\tau$ from $A$ to $A$ by sending $u$ to $\operatorname{Tr}_{A/R}(u)\cdot 1_{A} - u$, where $\operatorname{Tr}_{A/R}$ is the trace of the free algebra $A$ over $R$. Then, the map $\tau$ is an $R$-involution of $A$.
\end{proposition}

\begin{proof}
The trace map exists since \( A \) is a finite flat extension of \( R \). It is clear that \( \tau \) is \( R \)-linear, as \( \operatorname{Tr}_{A/R} \) is also \( R \)-linear. We consider the \( R \)-bilinear morphism 
\[
\begin{array}{llll} 
\beta:& A \times A & \rightarrow & A \\ 
& (u,v) & \mapsto & \tau(uv) - \tau(u)\tau(v) 
\end{array}
\]
and we aim to prove that \( \beta \) is zero.

First, we note that for all \( f \in R \), the commutativity of the following diagram holds:
\[
\xymatrix{ 
A \times A \ar[d] \ar[r]^\beta & A \ar[d] \\ 
A_f \times A_f \ar[r]_{\beta_f} & A_f 
}
\]
where the vertical arrows are the canonical localization morphisms, and \( \beta_f \) is the map sending \( (u,v) \) to \( \tau_f(uv) - \tau_f(u)\tau_f(v) \) with \( \tau_f  = \operatorname{Tr}_{A_f/R_f} 
\cdot 1_{A_f} - \operatorname{id}_{A_f} \). This commutativity follows from the fact that \( \operatorname{Tr}_{A_f/R_f} = {(\operatorname{Tr}_{A/R})}_f \) for all \( f \in R \). Moreover, since the map from \( A \) to the canonical direct limit of the \( A_f \)s is injective, it suffices to show that \( \beta_f \) is zero for all \( f \in R \) to conclude that \( \beta \) is zero.

Let \( f \in R \). By Proposition \ref{free}, we have that at every \( f \in A \), \( A_f = R_f \cdot 1_{A_f} \oplus R_f \cdot e_f \), where \( e_f \in A_f \) and \( e_f^2 = a_f \cdot e_f + b_f \cdot 1_{A_f} \) for some \( a_f, b_f \in R_f \). \\ The set \( \{ (1_{A_f}, e_f), (e_f, 1_{A_f}), (1_{A_f}, 1_{A_f}), (e_f, e_f) \} \) forms an \( R_f \)-basis for \( A_f \times A_f \) over \( R_f \), and it is clear that \( \beta_f \) is zero on each of the elements \( (1_{A_f}, e_f) \), \( (e_f, 1_{A_f}) \), and \( (1_{A_f}, 1_{A_f}) \).

We note that \( \tau_f(e_f) = -e_f + a_f\cdot 1_{A_f} \) and $ \tau_f(1_{A_f})=1_{A_f}$. Now, we compute \( \beta_f(e_f, e_f) \):

\begin{align*}
\beta_f(e_f, e_f) &= \tau_f(e_f^2) - \tau_f(e_f)^2 \\
&= \tau_f(a_f \cdot e_f + b_f \cdot 1_{A_f}) - \tau_f(e_f)^2 \\
&= a_f  \cdot (-e_f + a_f\cdot 1_{A_f} ) + b_f \cdot 1_{A_f} - \left(-e_f + a_f\cdot 1_{A_f}\right)^2 \\
&= a_f \cdot  (-e_f + a_f\cdot 1_{A_f} ) + b_f \cdot 1_{A_f} - (a_f^2 \cdot 1_{A_f} - 2a_f\cdot  e_f + e_f^2) \\
&= a_f \cdot (-e_f + a_f\cdot 1_{A_f} ) + b_f \cdot 1_{A_f} - (a_f^2 \cdot 1_{A_f}- 2a_f \cdot  e_f + a_f \cdot e_f + b_f \cdot 1_{A_f}) \\
&= 0.
\end{align*}

Thus, \( \beta_f \) is the zero map for all \( f \in R \), and consequently, \( \beta \) is also zero. This implies that \( \tau \) is an \( R \)-algebra morphism.

Now, we compute: for all $u \in A$,
\[
\tau^2(u) = \tau\left(\operatorname{Tr}_{A/R}(u)\cdot 1_{A} - u\right) = \operatorname{Tr}_{A/R}\left(\operatorname{Tr}_{A/R}(u) \cdot 1_{A}- u\right)  \cdot 1_{A} - \left(\operatorname{Tr}_{A/R}(u) \cdot 1_{A}- u\right) = u,
\]
since \( \operatorname{Tr}_{A/R}(1_{A}) = 2_{R} \).
\end{proof}

\noindent The previous lemma allows us to introduce the following definition.

\begin{definition}
Given \( A \) as a locally free \( R \)-algebra of rank 2, we define the \textsf{Galois conjugation of \( A \) over \( R \)} to be the involution \( \tau \) as described in Proposition \ref{galois}. We denote this map by \( \tau_A \).
\end{definition}

\noindent  We now aim to characterize the conditions under which a free algebra of rank 2 becomes a torsor under Galois conjugation. We will show that this is equivalent to the algebra being separable in terms of the coefficients of the polynomial defining the algebra. Additionally, since we are dealing with flat and finite extensions, separability is also equivalent to being unramified and to being \'etale.

\begin{proposition} \label{torsor}
Let $$ A =\frac{R[x]}{\langle x^2 +ax + b\rangle },$$
where $a, b \in R$. The following statements are equivalent. 
\begin{enumerate} 
\item $A$ is separable over $R$.
\item $a^2 - 4b$ is a unit in $R$. 
%\item $a^2 - 4b$ is a unit in $R_{\langle p \rangle}$, for all $p \neq 2$ and $a$ is invertible in $R_{\langle2 \rangle}$.
\item $A$ is a $\langle \tau_A \rangle$-Galois over $R$.
\item There exist $v,w\in R$ such that $2w=av$ and $-aw+2bv=1$. 
%\item $A$ is an Artin-Scheier algebra over $R$.
\end{enumerate} 
\end{proposition} 

\begin{proof}

The equivalence \((1) \Leftrightarrow (2)\) follows from \cite[Lemma 2]{micali}. Essentially, in characteristic zero, over a field, separability is equivalent to the discriminant of the polynomial \(x^2 + ax + b\) being invertible. In characteristic \(2\), over a field, separability is equivalent to the invertibility of the coefficient \(a\) in \(x^2 + ax + b\).

%The implication \((2) \Rightarrow (3)\) is straightforward.

%For \((3) \Rightarrow (2)\): This follows from Lemma \ref{charp}, since any prime ideal of \(R\) contains some \(p \cdot 1_R\) where \(p \in \mathbb{Z}\), and being a unit is equivalent to not being contained in any prime ideal. $ a^2 - 4b\notin p \cdot R $ for all $p$ 

The equivalence \((3) \Leftrightarrow (4)\): According to \cite[2.4.2]{Szeto} or \cite[Proposition 1.2, p. 80]{Meyer}, statement \((4)\) is equivalent to asserting that the left ideal \(I\) generated by \(\{ \tau_A(u) - u \mid u \in A \}\) equals \(R\). For any \(u \in A\), there exist \(r, s \in R\) such that \(u = r\overline{x} + s\), and we have:

\[
\tau_A(u) - u = (\tau_A(\overline{x}) - \overline{x}) r = (-2\overline{x} - a) r.
\]

Thus, \(I = (-2\overline{x} - a)R\). Therefore, \(I = R\) if and only if \(-2\overline{x} - a\) is invertible in \(R\). This implies the existence of \(c, d \in R\) such that 

\[
(-2\overline{x} - a)(c\overline{x} + d) = 1.
\]

Expanding this expression, we obtain:

\[
\begin{aligned}
1 &= (-2\overline{x} - a)(c\overline{x} + d) \\
&= -2c\overline{x}^2 + (-2d - ac)\overline{x} - ad \\
&= -2c(-a\overline{x} - b) + (-2d - ac)\overline{x} - ad \\
&= (ac - 2d)\overline{x} - ad + 2cb.
\end{aligned}
\]

This yields the equations \(2d = ac\) and \(-ad + 2bc = 1\).

For \((4) \Rightarrow (2)\): By multiplying the equation \(-aw + 2bv = 1\) by \(2bv\) and \(-aw\), and using the relationships \(2w = av\) and then by \(-aw + 2bv = 1\), we derive the equations:

\[
(-a^2 + 4b)bv^2 = 2bv \quad \text{and} \quad -(-a^2 + 4b)w^2 = -aw.
\]

Combining these results, we obtain:

\[
(-a^2 + 4b)(bv^2 - w^2) = 2bv - aw = 1.
\]

Now, for \((2) \Rightarrow (4)\): Assuming \(a^2 - 4b\) is a unit in $R$, there exists \(u \in R\) such that \((a^2 - 4b)u = 1\). By setting \(v = -2u\) and \(w = -au\), we have \(2w = av\) and \(-aw + 2bv = 1\), thus establishing the result.

%
%The equivalence $(1) \Leftrightarrow (2)$ is obtained easily. For the second part of the statement, we note that if $a^2-4b$ is invertible, so is $a\pm2b$. Thus, choosing $v= \pm b$, we conclude that any free separable algebra of rank $2$ over $R$ is Artin-Schreier.
\end{proof}

%\begin{proposition} \label{separable}
%Let $$ A = \frac{R[x]}{(x^2 + ax + b)},$$
%where $a, b \in R$.
%The following statements are equivalent: 
%\begin{enumerate} 
%\item $A$ is separable as an $R$-algebra;
%\item $a^2 - 4b$ is invertible in $R$. 
%\item $a^2 - 4b$ is invertible in $R_{(p)}$, for all $p \neq 2$ and $a$ is invertible in $R_{(2)}$;
%\end{enumerate}
%\end{proposition}
%\begin{proof}
%%$(1)\Leftrightarrow (2)$ follows from the fact that we are finite-dimensional.
%$(1)\Leftrightarrow (2)$ follows from Lemma \ref{charp} and the fact that the only algebraically closed field $\bar{k}$ for which we could have $A \otimes \bar{k}$ is not isomorphic to $\bar{k}^2$ are those of characteristic $2$, and given an algebraically closed field $\bar{k}$, each structural maps $\varphi: \mathbb{Z}\rightarrow \bar{k}$ sending $1_\mathbb{Z}$ to $1_{\bar{k}}$ factor through at least one residue field of a prime ideal of characteristic $2$.
%
%It is clear that $(4)\Rightarrow (3)$. The converse follows from the fact that an element is invertible in a ring if and only if it is not contained in any proper ideal. 
%\end{proof}
%From the previous result, the following corollary follows directly. We note that since a free algebra of rank $2$ over $R$ is of finite type over $R$, separable, unramifed and \'etale.

\noindent We can deduce the following corollary immediately by sending the class $\overline{x}$ of $x$ in $A$ to $2\overline{x}+a$ when $p\neq2$, and $\overline{x}$ to $a^{-1} \overline{x}$ when $p=2$ to obtain the stated isomorphisms.

\begin{corollary} \label{etale}
Let $p$ be a natural number and $R$ be a $\mathbb{Z}_{\langle p\rangle }$-algebra
 $$ A = \frac{R[x]}{\langle x^2 +ax + b\rangle},$$
where $a, b \in R$. Then, the following statements are equivalent: 
\begin{enumerate} 
\item $A$ is separable over $R$. 
\item \begin{itemize} 
    \item If \(p \neq 2\), then \(A \simeq \frac{R[y]}{\langle y^2 - (a^2 - 4b) \rangle}\) with \(\overline{y} = 2\overline{x} - a\), and  \(a^2 - 4b\) is invertible in \(R\).
    \item If \(p = 2\), then \(A \simeq \frac{R[y]}{\langle y^2 - y + a^{-2}b \rangle}\) with \(\overline{y} = a^{-1} \overline{x}\), and \(a\) is invertible in \(R\).
\end{itemize}

\end{enumerate}
\end{corollary}

We consider \( R \) to be a \( \mathbb{Z}_{(p)} \)-algebra. According to Proposition \ref{etale}, any separable \( R \)-free algebra of rank \( 2 \) is Artin-Schreier when \( p = 2 \). Furthermore, when \( p \neq 2 \), any separable \( R \)-free algebra of rank \( 2 \) is radical.

\section{The structure of the functor of isomorphisms} 
Since the goal of this paper is to study free algebras up to isomorphism, we introduce notation for the sets of isomorphisms and automorphisms between two free algebras of rank \( 2 \).

\begin{definition}
Let \( a, b, c, d \in S \).
\begin{enumerate} 
\item For any \( S \)-algebra \( R \), we define $\operatorname{Iso}_S((a,b), (c,d))(R)$ is the set of \( R \)-algebra isomorphisms from \(\frac{R[x]}{\langle x^2 + ax + b\rangle }\) to \(\frac{R[y]}{\langle y^2 + cy + d \rangle }\).

\item When \((a,b) = (c,d)\), for any \( S \)-algebra \( R \),  we denote \(\operatorname{Iso}_S((a,b), (a,b))(R)\) by \(\operatorname{Aut}_S(a,b)(R)\).
\end{enumerate}
\end{definition}

In the following proposition, we establish that the isomorphisms between two free algebras of rank \( 2 \) define an affine scheme. Additionally, we show that the automorphisms of a free algebra of rank \( 2 \) defines an affine group scheme.
\begin{proposition} \label{uni}
Let \( a, b, c, d \in S \). Then:

\begin{enumerate}
\item  For any \( S \)-algebra \( R \),  we have a bijection between  \(\operatorname{Iso}_S((a,b), (c,d)))(R)\) and 
\[
\operatorname{Hom}_{S\text{-}\bf{Al}} \left( \frac{S[x, 1/x, y]}{\langle 2y - cx + a, -dx^2 + y^2 + ay + b\rangle} , R \right).
\]
This bijection is natural in $R$. We identify $\operatorname{Iso}_S((a,b), (c,d))$ with the affine scheme $\operatorname{Spec}\left( \frac{S[x, 1/x, y]}{\langle 2y - cx + a, -dx^2 + y^2 + ay + b\rangle} \right)$. \\
\noindent  In particular, for any \( S \)-algebra \( R \), we have 
\[
\operatorname{Iso}_S((a,b), (c,d))(R) \simeq \{ (v,w) \in R \times R^\times \mid 2v = cw - a \text{ and } -dw^2 + v^2 + av + b = 0 \}.
\]

\item  For any \( S \)-algebra \( R \),  we have an isomorphism between  \(\operatorname{Aut}_S(a,b)(R)\) and 
\[
\operatorname{Hom}_{S\text{-}\bf{Al}} \left( \frac{S[x, 1/x, y]}{\langle 2y - a(x-1), b(1-x^2) + y^2 + ay\rangle }, R \right).
\]

\noindent $\frac{S[x, 1/x, y]}{\langle 2y - a(x-1), b(1-x^2) + y^2 + ay\rangle }$ is the Hopf algebra whose comultiplication is defined by mapping \(\overline{x}\) to \(\overline{x} \otimes \overline{x}\) and \(\overline{y}\) to \(\overline{x} \otimes \overline{y} + \overline{y} \otimes \overline{1}\). The antipode maps \(\overline{x}\) to \(\overline{x}^{-1}\) and \(\overline{y}\) to \(-\overline{x}^{-1} \overline{y}\), while the counit maps \(\overline{x}\) to \(\overline{1}\) and \(\overline{y}\) to \(\overline{0}\). We identify $\operatorname{Aut}_S(a,b)$ with the affine group scheme $\operatorname{Spec}\left( \frac{S[x, 1/x, y]}{\langle 2y - a(x-1), b(1-x^2) + y^2 + ay\rangle } \right)$. \\

\noindent  Moreover, for any \( S \)-algebra \( R \), we have 
\[
\operatorname{Aut}_S(a,b)(R) \simeq \{ (v,w) \in R \times R^\times \mid 2v = a(w - 1) \text{ and } b(1 - w^2) + v^2 + av = 0 \}.
\]
\end{enumerate}
\end{proposition}
\begin{proof}
\begin{enumerate}
\item Let \( R \) be an \( S \)-algebra. By \cite[Lemma 1]{Kitamura}, we know that the following statements are equivalent:

\begin{enumerate}
\item  \( \varphi \) is an isomorphism of \( R \)-algebras from \( \frac{R[x]}{\langle x^2 + ax + b\rangle} \) to \( \frac{R[y]}{\langle y^2 + cy + d \rangle} \).
\item \( \varphi \) is a morphism of \( R \)-algebras from \( \frac{R[x]}{\langle x^2 + ax + b\rangle} \) to \( \frac{R[y]}{\langle y^2 + cy + d\rangle } \) such that \( \varphi(\overline{x}) = w \overline{y} + v \) where \( v \in R \) and \( w \in R^\times \), satisfying the conditions \( 2v = cw - a \) and \( -dw^2 + v^2 + av + b = 0 \).
\end{enumerate} 

From this, we deduce
\[
\operatorname{Iso}_S((a,b), (c,d))(R) \simeq \mathcal{S}(R),
\]
where \( \mathcal{S}(R) := \{ (v,w) \in R \times R^\times \mid 2v = cw - a \text{ and } -dw^2 + v^2 + av + b = 0 \} \).

We define 
\[
A := \frac{S[x, 1/x, y]}{\langle 2y - cx + a, -dx^2 + y^2 + ay + b \rangle}.
\]
It is straightforward to show that the map
\[
\begin{array}{lrll} 
\Psi_R: & Hom_S(A, R) & \rightarrow & \mathcal{S}(R) \\
& f & \mapsto & (f(\overline{y}), f(\overline{x}))
\end{array}
\]
is a bijection.

\item Next, when \( (a,b) = (c,d) \), the composition of two automorphisms \( \varphi \) and \( \varphi' \) such that \( \varphi(\overline{x}) = w \overline{x} + v \) and \( \varphi'(\overline{x}) = w' \overline{x} + v' \) yields the automorphism \( \varphi \circ \varphi' \) given by \( \varphi \circ \varphi'(\overline{x}) = ww' \overline{x} + w v' + v \). This construction induces a group law on \( \mathcal{S}(R) \):
\[
(v, w) \star (v', w') = (wv' + v, ww'),
\]
for any \( R \)-algebra \( S \) and \( (v, w) \in \mathcal{S}(R) \). The identity element for this law is \( (0, 1) \), and the inverse is \( (-w^{-1} v, w^{-1}) \). From this, we deduce the comultiplication, antipode, and counit that endow \( A \) with the structure of a Hopf algebra as described in the statement.
\end{enumerate}
\end{proof}
The proposition computed the isomorphism between points in the moduli space under consideration. Notably, we observe that the automorphisms of a point correspond naturally to the functor of points of a closed subgroup scheme of \( \mathbb{G}_m^{S} \times \mathbb{G}_a^S \). In the following remark, we will elaborate on this observation.
\begin{remark}
We offer some observations on the nature of \( \operatorname{Aut}_S(a,b) \) when we impose further restrictions on \( S \) or the coefficients \( a \) and \( b \). Specifically, we identify the group structure when \( S \) is a field across various scenarios.

\begin{enumerate}
\item If \( S \) is a ring of characteristic \( 2 \) and \( a \) is invertible in \( S \), the Hopf algebra \( \frac{S[y]}{\langle y^2 + ay\rangle } \) represents the functor \( \operatorname{Aut}_S(a,b) \). This is a deformation of the constant group scheme of order \( 2 \). In particular, when \( S \) is an algebraically closed field of characteristic \( 2 \) and \( a \neq 0 \), we have \( \operatorname{Aut}_S(a,b) \cong \left( \frac{\mathbb{Z}}{2 \mathbb{Z}} \right)_S \), the constant group scheme of order \( 2 \).

\item If \( S \) is a ring of characteristic \( 2 \) and \( a = 0 \), the Hopf algebra representing \( \operatorname{Aut}_S(a,b) \) becomes \( \frac{S[x, 1/x, y]}{\langle b(1-x)^2 + y^2\rangle} \), which is a deformation of the product \( \mathbb{G}_m^S \times \alpha_2^S \). In particular, when \( S \) is an algebraically closed field of characteristic \( 2 \) and \( a = 0 \), we find \( \operatorname{Aut}_S(a,b) \simeq \mathbb{G}_m^S \times \alpha_2^S \), the direct product of \( \mathbb{G}_m^S \) and \( \alpha_2^S \).

\item If \( 2 \) is a unit in \( S \), the Hopf algebra representing \( \operatorname{Aut}_S(a,b) \) takes the form \( \frac{S[x, 1/x]}{\langle (a^2 - 4b)(x^2 - 1)\rangle } \). Specifically, when \( S \) is a field of characteristic not equal to \( 2 \) and \( a^2 - 4b = 0 \), we have \( \operatorname{Aut}_S(a,b) \cong \mathbb{G}_m^S \). Otherwise, when \( S \) is a field of characteristic not equal to \( 2 \) and \( a^2 - 4b \neq 0 \), we find \( \operatorname{Aut}_S(a,b) \cong \mu_2^S \).
\end{enumerate}
\end{remark}
We observe that under certain assumptions, a free algebra of rank \( 2 \) or the isomorphism groups can be expressed with fewer parameters. We illustrate this with two examples in the following results.

\begin{corollary}\label{propd}
Let \( a, b, c, d \in S \) and \( R \) be a ring in which \( 2 \) is a regular element and the characteristic is not \( 2 \). We have 
\[
\operatorname{Iso}_S((a,b), (c,d))(R) \simeq \{ w \in R^\times \mid a^2 - 4b = w^2(c^2 - 4d) \text{ and } 2 \mid (cw - a) \}.
\]
Assuming further that \( 2 \) is a unit, we obtain 
\[
\frac{R[x]}{\langle x^2 + ax + b \rangle} \simeq \frac{R[y]}{\langle y^2 - (a^2 - 4b) \rangle}.
\]
\end{corollary}

\begin{proof} 
When \( 2 \) is a regular element in \( R \), it follows that \( 4 \) is also a regular element. The equation \( 4v^2 + 4av + 4b - 4dw^2 = 0 \) is equivalent to \( v^2 + av + b - dw^2 = 0 \). 

From the equations $2v = cw - a$ and $-dw^2 + v^2 + av + b = 0$, we deduce that:
\[
\begin{aligned}
0 &= 4v^2 + 4av + 4b - 4dw^2 \\
&= (cw - a)^2 + 2a(cw - a) + 4b - 4dw^2 \\
&= c^2 w^2 - 2cwa + a^2 + 2acw - 2a^2 + 4b - 4dw^2 \\
&= (c^2 - 4d)w^2 - (a^2 - 4b).
\end{aligned}
\]

The equivalence shows that our initial conditions hold. 

The second part of the lemma is clear since completing the square is feasible under the assumption, which concludes the proof. 
\end{proof}

\begin{corollary} \label{regular}
Let \( a, b, c, d \in S \) and \( R \) be a ring with the following properties:
\begin{itemize} 
\item \( 2 \) is a regular element in \( R \),
\item \( 2 \) is either a prime element or a unit in \( R \).
\end{itemize}
Under these conditions, we have
\[
\operatorname{Iso}_S((a,b), (c,d))(R) \simeq \{ w \in R^\times \mid a^2 - 4b = w^2(c^2 - 4d) \}.
\]
\end{corollary} 

\begin{proof}
If \( 2 \) is invertible in \( R \), the result follows trivially. 

When \( 2 \) is a prime element in \( R \), we note that if \( a^2 - 4b = w^2(c^2 - 4d) \), then we can express this as 
\[
4(b - w^2d) = a^2 - w^2c^2 = (a - wc)(a + wc).
\]
Since \( 2 \) is prime, it must divide either \( a - wc \) or \( a + wc \). 

By replacing \( w \) with \( -w \) if necessary, we can conclude that the result holds. 
\end{proof}

\begin{example} \label{counter}
If \( 2 \) is a regular element but neither a prime nor a unit in \( R \), the corollary may not hold. Consider \( R = \mathbb{Z}[\sqrt{5}] \), \( A = R[x]/(x^2 + x + 1) \), and \( B = R[y]/(y^2 + \sqrt{5}y + 2) \). In this case, \( 2 \) is irreducible but not prime in \( R \). Both polynomials \( x^2 + x + 1 \) and \( y^2 + \sqrt{5}y + 2 \) share the same discriminant, \( -3 \).

\noindent The existence of an isomorphism from \( A \) to \( B \) is equivalent to finding \( w \in R^\times \) such that \( a^2 - 4b = w^2(c^2 - 4d) \) and \( 2 \mid (cw - a) \), as shown in Corollary \ref{propd}. This implies that \( w^2 = 1 \).

However, the elements \( -1 \pm \sqrt{5} \) are irreducible in \( R \), which means that \( 2 \) cannot divide either \( -1 + \sqrt{5} \) or \( -1 - \sqrt{5} \). Consequently, \( A \) and \( B \) cannot be isomorphic according to Corollary \ref{propd}.
\end{example}

%
%Considering two $R$-algebras, $A$ and $B$, the following assertions are equivalent:
%\begin{enumerate}
%\item $A$ and $B$ are isomorphic as $R$-algebras.
%\item $A_{(p)}$ and $B_{(p)}$ are isomorphic as $R$-algebras for all prime numbers $p$.
%\end{enumerate}
\noindent We study the isomorphism classes locally at all prime numbers of $\mathbb{Z}$. 
%The following corollary follows directly from Corollary \ref{regular}, since $2$ is invertible in $\mathbb{Z}_{(p)}$ for all $p \neq 2$. 
As shown in Corollary \ref{etale}, separable algebras of rank $2$ over $R$ are either radical or Artin-Schreier over $R$. Let's consider these cases.
\subsubsection*{The Radical Case}
\begin{itemize}
\item The functor \( \operatorname{Iso}_S((0,a), (0,b)) \) is represented by an affine scheme over \( S \), specifically the spectrum of the \( S \)-algebra:
\[
\frac{S[x, 1/x, y]}{\langle 2y, -bx^2 + y^2 + a \rangle}.
\]

\item The functor \( \operatorname{Aut}_S(0,a) \) is represented by an affine group scheme over \( S \), given by the spectrum of the Hopf algebra:
\[
A = \frac{S[x, 1/x, y]}{(2y, a(1 - x^2) + y^2)}
\]
with a coalgebra structure as described in Proposition \ref{uni}.

\begin{itemize}
\item When \( 2 \) is a regular element and \( (0,a) \) defines a separable algebra over \( R \), \( a \) is invertible (see Proposition \ref{torsor}). In this case, we have 
\[
A = \frac{S[x]}{\langle x^2 - 1 \rangle}.
\]
This represents the Hopf algebra of \( \mu_2^S \) when \( S \) is a field of characteristic not equal to \( 2 \).

\item If \( S \) is a field of characteristic \( 2 \), then we have 
\[
A = \frac{S[x, 1/x, y]}{\langle -ax^2 + y^2 + a \rangle}.
\]
Assuming further that \( a = 0 \), we obtain 
\[
A = \frac{S[x, 1/x, y]}{\langle y^2 \rangle},
\]
which represents the Hopf algebra of the group scheme \( \mathbb{G}_m^S \times \alpha_2^S \). If \( a \neq 0 \), then 
\[
A = \frac{S[x, 1/x, y]}{\langle y^2 - a(x - 1)^2 \rangle},
\]
which represents the Hopf algebra of a deformation of the product \( \mathbb{G}_m^S \times \alpha_2^S \).
\end{itemize}
\end{itemize}
\subsubsection*{The Artin-Schreier Case}
\begin{itemize}
\item The functor \( \operatorname{Iso}_S((-1,a), (-1,b)) \) is represented by an affine scheme over \( S \), specifically the spectrum of the \( S \)-algebra:
\[
\frac{S[x, 1/x, y]}{\langle 2y + x - 1, (4b + 1)(y^2 - y) + a - b \rangle}.
\]

\item The functor \( \operatorname{Aut}_S(-1,a) \) is represented by an affine group scheme over \( S \), given by the spectrum of the Hopf algebra:
\[
A = \frac{S[x, 1/x, y]}{\langle 2y + x - 1, (4a + 1)(y^2 - y) \rangle}
\]
with a Hopf algebra structure as described in Proposition \ref{uni}.
When \( S \) is a field of characteristic \( 2 \), this simplifies to 
\[
A = \frac{S[y]}{\langle y^2 - y \rangle},
\]
which represents the constant group scheme \( \left( \frac{\mathbb{Z}}{2\mathbb{Z}} \right)_S \).

 When \( (-1,a) \) defines a separable algebra over \( R \) and \( 4a + 1 \) is invertible (see Proposition \ref{torsor}), we have 
\[
A = \frac{S[x, 1/x, y]}{\langle 2y + x - 1, y^2 - y \rangle}.
\]
\end{itemize}

The following corollary follows from Corollary \ref{regular}, since \( 2 \) is a unit in any \( \mathbb{Z}_{\langle p\rangle } \)-algebra for any prime number \( p \) not equal to \( 2 \).
\begin{corollary} \label{pneq2}
Let \( p \) be a prime number not equal to \( 2 \), \( S \) a \( \mathbb{Z}_{\langle p\rangle} \)-algebra, and let \( a, b, c, d \in S \). 

\begin{enumerate} 
\item The functor of points of the scheme \( \operatorname{Iso}_{S}((a,b), (c,d)) \) is represented by an affine scheme over \( S \), specifically the spectrum of the \( S \)-algebra:
\[
\frac{S[x, 1/x]}{\langle (a^2 - 4b) - x^2(c^2 - 4d) \rangle}.
\]

\item The functor of points of group scheme \( \operatorname{Aut}_{S}(a,b) \) is represented by an affine group scheme over \( S \), given by the spectrum of the Hopf algebra:
\[
A = \frac{S[x, 1/x]}{\langle (a^2 - 4b) - x^2(a^2 - 4b) \rangle}
\]
with comultiplication sending \( \overline{x} \) to \( \overline{x} \otimes \overline{x} \), the antipode sending \( \overline{x} \) to \( \overline{x}^{-1} \), and the counit sending \( \overline{x} \) to \( \overline{1} \). In particular, when \( a^2 - 4b \) is invertible in \( S \), we have $A = \frac{S[x, 1/x]}{\langle 1 - x^2 \rangle},$ which represents the Hopf algebra of \( \mu_2^S \).
\end{enumerate}
\end{corollary}

\section{The structure of the moduli spaces}

%\subsection{Free algebras of rank 2} 
\subsection{The structure of the moduli spaces at a given algebra}
In this section, we explore the various moduli spaces evaluated at a given algebra, writing them as quotients of a ring by a specific group action, which we will describe in the following definition. Given an action \( \star \) of a group \( G \) on a set \( B \), we denote the orbit of an element \( b \in B \) as \( \operatorname{o}_\star(b)=\{ g \cdot b | g \in G\} \). The set of all orbits under this action is denoted as \( \sbfrac{B}{G}{\star} \).

\begin{deflem} \label{action} \text{}
Let $R$ be an $S$-algebra. 
\begin{enumerate} 
\item We define the binary operation $\smallstar$ on the set $R^\times \times R$ as follows: 
$$(w,v)\smallstar(w',v')=(ww' , wv' +v),$$
for all $(w, v)$ and $(w', v') \in R^\times \times R$.

\noindent The set $R^\times \times R$ equipped with this operation forms a group with unit element $(1,0)$ and such that the inverse of an element $(w,v)\in R^\times \times R$ is $(w^{-1} ,-w^{-1}v)$. This group is an external semi-direct product of $R$ by $R^\times$, denoted as $R^\times \ltimes_\smallstar R$.
\item We define the right group action $\star$ of the group $R^\times \ltimes_\smallstar R$ on the Cartesian product $R \times R$ as follows: 
$$(a,b)\star(w,v)=(w^{-1}(2v+a) , w^{-2}(v^2+av+b)),$$
for all $(w,v) \in R^\times \ltimes_\smallstar R$ and $(a,b) \in R\times R$. 
\item We define ${\bf S} (R)$ to be the set $\{ (a, b) \in R \times R| a^2-4b\in R^\times\}$, also equal to the set $\{ (a,b) \in R \times R| \exists (r,t) \in R\times R: 2t=ar \wedge -at+2br=1 \}$.\\
 The action $\star$ restricts to an action of $R^\times \ltimes_\smallstar R$ on ${\bf S}(R)$.\\
  We still denote this action by $\star$. 
 % \item We define ${\bf E}(R)$ as the set $$.\\ The action $\star$ restricts to an action of $R^\times \ltimes_\smallstar R$ on ${\bf E}(R)$. We still denote this action by $\star$. 
 \item  We define ${\bf R}(R)$ as the set $\{ (a, b) \in R \times R| \exists r \in R: a=2r \}$.\\
The action $\star$ restricts to an action of $R^\times \ltimes_\smallstar R$ on ${\bf R}(R)$.\\
 We still denote this action by $\star$. 
 \item We define ${\bf AS}(R)$ as the set $\{ (a, b) \in R \times R| \exists r \in R: 2r + a \in R^\times \}$.\\
The action $\star$ restricts to an action of $R^\times \ltimes_\smallstar R$ on ${\bf A}(R)$.\\
 We still denote this action by $\star$.  
\item We define a group action $\diamond$ of the multiplicative group $(R^\times,\cdot)$ on $R$ such that $w \diamond a = w^2  a$, for all $w\in R^\times$ and $a\in R$. \\
%We denote $o_\diamond(a)$ the orbit of $a$ with respect to the action $\diamond$, and $\frac{R}{(R^\times)^\diamond}$ the set of orbits for this action. 
This action restricts to an action of $ (R^\times,\cdot)$ on $R^\times$ that we still denote by $\diamond$. 

%We denote $o_\ast(a)$ the orbit of $a$ with respect to the action $\ast$, and $\frac{R}{(P(R))^\ast}$ the set of orbits for this action.
\item We define ${\bf T}(R)$ as the set $\{(w, v) | w\in R^\times, v\in R \wedge 2v=0\}$. ${\bf T}(R)$ is a subgroup of $R^\times \ltimes_\smallstar R$. We note that in characteristic $2$, ${\bf T}(R)=R^\times \ltimes_\smallstar R$.
\item We define an action $ \pentagram$ of the group ${\bf T}(R)$ on $R$ by $$   a\pentagram (w, v) =w^{-2} (v^2 +a)$$ for all $(w, v)\in {\bf T}(R)$. 

\item We define ${\bf H}(R)$ as the set $\{(1-2v, v) \in R \times R| 1-2v \in R^\times\}$. ${\bf H}(R)$ is a subgroup of $R^\times \ltimes_\smallstar R$. 
\item We define a right group action $\ast$ of the group ${\bf H}(R)$ on the additive group $(R,+)$ as follows: $$a \ast (1-2v, v)  =(1-2v)^{-2} (v^2 -v+a)$$ for all $v\in R$ such that $(1-2v, v)\in {\bf H} (R)$.\\ 
\item We define ${\bf L}(R)$ as the set $\{a\in R| 1-4a\in R^\times\}$. \\
This action restricts to an action of ${\bf H}(R)$ on ${\bf L}(R)$.\\ 
We still denote this action by $\ast$.
%We denote $o_\pentagram(a)$ the orbit of $a$ with respect to the action $\pentagram$, and $\frac{R}{(F(R))^\pentagram}$ the set of orbits for this action.
\end{enumerate}
\end{deflem}
\begin{proof}
\begin{enumerate}
\item Let \((w, v)\) and \((w', v') \in {R}^\times \times {R}\). We have:

\[
\begin{array}{lll} 
((w,v) \smallstar (w',v')) \smallstar (w'',v'') &=& (ww', wv' + v) \smallstar (w'',v'') \\ 
&=& (ww'w'', ww'v'' + wv' + v) \\ 
&=& (ww'w'', w(w'v'' + v') + v) \\ 
&=& (w,v) \smallstar (w'w'', w'v'' + v') \\ 
&=& (w,v) \smallstar ((w',v') \smallstar (w'',v''))
\end{array}
\]
We also have:
\[
\begin{array}{lll}  (1,0) \smallstar (w,v) &=& (1w, 1 v + 0) = (w,v) \\
&=& (w  1, 0w  + v) = (w,v) \smallstar (1,0)
\end{array}
\]
and:
\[
\begin{array}{lll}  
(w^{-1}, -w^{-1}v) \smallstar (w,v) &=& (w^{-1}w, w^{-1}v - w^{-1}v) = (1,0) \\ 
&=& (ww^{-1}, w(-w^{-1}v) + v) \\ 
&=& (w,v) \smallstar (w^{-1}, -w^{-1}v)
\end{array}
\]

This proves that \({R}^\times \ltimes_\smallstar {R}\) is a group.
\item Let \((w, v)\), \((w', v') \in {R}^\times \times {R}\) and \((a,b) \in {R} \times {R}\). We have:

\[
\begin{array}{lll} 
&& ((a,b) \star (w,v)) \star (w', v') \\ 
&=& (w^{-1}(2v+a), w^{-2}(v^2 + av + b)) \star (w', v') \\ 
&=& \left( w'^{-1}\left(2v' + w^{-1}(2v + a)\right), w'^{-2}\left(v'^2 + (w^{-1}(2v + a))v' + w^{-2}(v^2 + av + b)\right) \right) \\ 
&=& \left( w'^{-1} w^{-1}(2wv' + 2v + a), w'^{-2} w^{-2}(w^2v'^2 + (w(2v + a))v' + v^2 + av + b) \right) \\ 
&=& \left( (ww')^{-1}(2(wv' + v) + a), (ww')^{-2}((wv' + v)^2 + a(wv' + v) + b) \right) \\ 
&=& (a,b) \star (ww', wv' + v) \\ 
&=& (a,b) \star ((w, v) \smallstar (w', v')).
\end{array}
\]
Additionally, we have:
\[
\begin{array}{lll} 
&& (a,b) \star (1,0) = \left( 1^{-1}(2 \times 0 + a), 1^{-2}(0^2 + a \times 0 + b) \right) = (a, b)
\end{array}
\]
This proves that \(\star\) defines a group action.\\
\item If \((a, b) \in \mathbf{S}(R)\), meaning that \(a^2 - 4b \in {R}^\times\), we have:
\[
\begin{array}{lll} 
&& (w^{-1}(2v + a))^2 - 4w^{-2}(v^2 + av + b) \\ 
&=& w^{-2}(4v^2 + 4va + a^2) - 4w^{-2}(v^2 + av + b) \\ 
&=& w^{-2}(a^2 - 4b).
\end{array}
\]
Since \(w^{-2}(a^2 - 4b) \in {R}^\times\) (as \(w \in {R}^\times\)), it follows that \((w^{-1}(2v + a), w^{-2}(v^2 + av + b)) \in \mathbf{S}(R)\).
This proves that \(\star\) induces an action on \(\mathbf{S}(R)\).
%\item If \((a, b) \in \mathbf{S}(R)\), meaning that there exists \( (r,t) \in R\times R\) such that \( 2t=ar \), and \( -at+2br=1 \). We set $r'= rw^{2}$ and $t'=w(vr+t)$  we have:
%$$w^{-1}(2v+a) r'= rw^{2}w^{-1}(2v+a) = 2w^{2}w^{-1}(vr+t)=2t' $$
%and
%$$\begin{array}{lll} && -w^{-1}(2v+a)t' + 2w^{-2}(v^2+av+b) r'\\
%&& -(vr+t)(2v+a)  + 2(v^2+av+b)r \\
%&=&  -(r(2v^2+av)+(2vt+at))  + 2(v^2+av+b)r\\
%&=&  -(r(2v^2+av)+(var+2br -1))  + 2(v^2+av+b)r =1
% \end{array}$$
%(w^{-1}(2v+a) , w^{-2}(v^2+av+b))

It follows that \((w^{-1}(2v + a), w^{-2}(v^2 + av + b)) \in \mathbf{S}(R)\).
This proves that \(\star\) induces an action on \(\mathbf{S}(R)\).\\
\item If \((a, b) \in \mathbf{R}(R)\), meaning that there exists $r \in R$ such that $a=2r$, we have:
$$w^{-1}(2v + a)= w^{-1}(2v + 2r)=2w^{-1}(v + r). $$
This proves that \(\star\) induces an action on \(\mathbf{R}(R)\).
\item If \((a, b) \in \mathbf{AS}(R)\), meaning that there exists $r \in R$ such that $2r+a\in R^\times$. We set $r'= w^{-1} (r-v) \in R$. Then we have:
$$ 2 r' + w^{-1}(2v + a)=  2w^{-1} (r-v)  + w^{-1}(2v + a)= w^{-1}(2r + a). $$
This proves that \(\star\) induces an action on \(\mathbf{AS}(R)\).
\item is clear.

\item Let \((w,v)\) and \((w',v') \in \mathbf{T}(R)\). This means \(w, w' \in R^\times \), \(v, v' \in R\) and \(2v = 2v' = 0\). We have:

\[
(w, v) \smallstar (w', v') = (ww', wv' + v).
\]
with $2( wv' + v)= 0$. Thus, $(w, v) \smallstar (w', v') \in \mathbf{T}(R)$.
Clearly, \((1, 0) \in \mathbf{T}(R)\) and \((w^{-1}, -w^{-1}v) \in \mathbf{T}(R)\) since \(-2w^{-1}v= 0\). This shows that \(\mathbf{T}(R)\) forms a group.

\item is clear, since $(v+v')^2 =v^2 +v'^2$ whenever $2v=2v'=0$. 

\item Let \((w,v), (w',v') \in \mathbf{H}(R)\). This means \(v, v' \in R\), \(w = 1 - 2v\), \(w' = 1 - 2v'\) and $w, w'\in R^\times$. We have:
\[
\begin{array}{lll} 
(w, v) \smallstar (w', v') &=& ((1 - 2v)(1 - 2v'), (1 - 2v)v' + v) \\ 
&=& \left( 1 - 2\left((1 - 2v)v' + v\right), (1 - 2v)v' + v \right)  
\end{array}
\]
Thus, \((w, v) \smallstar (w', v') \in \mathbf{H}(R)\). Clearly, \((1, 0) \in \mathbf{H}(R)\).\\
To prove that the inverse of \((w,v)\) in \(\mathbf{H}(R)\) exists, we need to show that \(w^{-1} = 1 - 2(-w^{-1}v)\). This follows trivially from the definition of \(w^{-1}\), since we have:
\[
w^{-1} w = w^{-1}(1 - 2v) = 1.
\]

\item Let \((w,v), (w',v') \in \mathbf{H}(R)\) and \(a \in R\). This means \(v, v' \in R\), \(w = 1 - 2v\), \(w' = 1 - 2v'\), and \(w, w' \in \mathbb{R}^\times\).

\[
\begin{array}{lll} 
&& (a \ast (w,v)) \ast (w', v') \\ 
&=& (1-2v)^{-2} (v^2 - v + a) \ast (w', v') \\ 
&=& (1 - 2v')^{-2} \left(v'^2 - v' + (1 - 2v)^{-2} (v^2 - v + a)\right) \\ 
&=& w'^{-2} \left(w^{-2} (w^2 v'^2 + 2wv'v + v^2) + w^{-2} (-wv' - v + a)\right) \\ 
&=& w'^{-2} w^{-2} \left((wv' + v)^2 - (wv' + v) + a\right) \\ 
&=& (1 - 2(wv' + v))^{-2} \left((wv' + v)^2 - (wv' + v) + a\right) \\ 
&=& a \ast (ww', wv' + v) \\ 
&=& a \ast ((w, v) \smallstar (w', v'))
\end{array}
\]

Clearly, \(a \ast (1, 0) = a\). This shows that \(\ast\) defines a right group action.

\item Let \(a \in \mathbf{L}(R)\), meaning \(1 - 4a \in {R}^\times\), and let \((w,v) \in \mathbf{H}(R)\), meaning \(1 - 2v \in {R}^\times\). We have:

\[
a \ast (w,v) = (1 - 2v)^{-2} (v^2 - v + a),
\]

and 

\[
\begin{array}{lll} 
1 - 4(a \ast (w,v)) &=& (1 - 2v)^{-2} \left((1 - 2v)^{2} - 4(v^2 - v + a)\right) \\ 
&=& (1 - 2v)^{-2} \left((1 - 4v + 4v^2) - 4(v^2 - v + a)\right) \\ 
&=& (1 - 2v)^{-2} (1 - 4a).
\end{array} 
\]

Since \(1 - 4(a \ast (w,v)) \in {R}^\times\), we conclude that \(a \ast (w,v) \in \mathbf{L}(R)\). This proves that \(\ast\) also defines an action on \(\mathbf{L}(R)\).

%$\{(w, v) \in R^\times \times R| 2v=0\}$
% $(1-2v)^{-2} (v^2 -v+a) $
%${\bf L}(R):=\{a\in R| 1-4a\in R^\times\}$
%
%
%$(w,v) \smallstar (g,h)\smallstar (w'',v'')=(wgw'', w(gv'' + h) + v)$
%
%$g=1-2h$ and $wg w'' = 1-2 (w(gv'' + h) + v)$
%We want to prove that $w^{-1}$(w^{-1} ,-w^{-1}v)$
%$\{(1-2v, v) \in R \times R| 1-2v \in R^\times\}$, $(ww' , wv' +v)$, $(a,b)\star(w,v)=(w^{-1}(2v+a) , w^{-2}(v^2+av+b)),$. 

\end{enumerate}
\end{proof} 
From Proposition \ref{radical}, Proposition \ref{AS}, Proposition \ref{uni}, Proposition \ref{pneq2} and the notation of Definition \ref{action}, we obtain
\begin{proposition} \label{propmain}
Let \( R \) be a \( S \)-algebra. The set \(\af{{\sf P}}{2}{R}{S}\) is in bijection with:

\begin{enumerate} 
    \item \text{}
    \begin{itemize}
        \item \( \displaystyle \sbfrac{R \times R}{R^\times \ltimes_\smallstar R}{\star} \), when \({\sf P} = {\sf F}\).
        \item \(\displaystyle \sbfrac{{\bf S}(R)}{R^\times \ltimes_\smallstar R}{\star} \), when \({\sf P} = {\sf SF}\).
        \item \(\displaystyle \sbfrac{{\bf R}(R)}{R^\times \ltimes_\smallstar R}{\star} \), when \({\sf P} = {\sf R}\).
        \item \( \displaystyle \sbfrac{{\bf AS}(R)}{R^\times \ltimes_\smallstar R}{\star} \), when \({\sf P} = {\sf AS}\).
    \end{itemize}
    These bijections all send \([R[x]/\langle x^2 + ax + b\rangle ]_{{\sf P}, 2, R}\) to \(o_\star(a, b)\).

    \item \text{}
    \begin{itemize}
        \item \(\displaystyle \sbfrac{R}{{\bf T}(R)}{\pentagram} \), when \({\sf P} = {\sf R}\).
        \item \( \displaystyle \sbfrac{R}{R^\times}{\diamond} \), when \({\sf P} \in \{ {\sf F}, {\sf R}\}\) and \(S\) is a \(\mathbb{Z}_{\langle p \rangle}\)-algebra, where \(p\) is a prime distinct from \(2\).
        \item \(\displaystyle \sbfrac{R^\times}{R^\times}{\diamond} \), when \({\sf P} = {\sf SR}\) and \(S\) is a \(\mathbb{Z}_{\langle p \rangle}\)-algebra, where \(p\) is a prime distinct from \(2\).
    \end{itemize}
    These bijections all send \([R[x]/\langle x^2 + a\rangle ]_{{\sf P}, 2, R}\) to \(o_\diamond(a)\).

    \item \text{}
    \begin{itemize} 
        \item \( \displaystyle \sbfrac{R}{{\bf H}(R)}{\ast} \), when \({\sf P} = {\sf AS}\).
        \item \( \displaystyle \sbfrac{{\bf L}(R)}{{\bf H}(R)}{\ast} \), when \({\sf P} = {\sf SAS}\).
    \end{itemize}
    This bijection sends \([R[x]/\langle x^2 - x + a\rangle ]_{{\sf P}, 2, R}\) to \(o_\ast(a)\).
\end{enumerate}

In particular, there exists a bijection from:

\begin{itemize} 
    \item \(\displaystyle \sbfrac{{\bf AS}(R)}{R^\times \ltimes_\smallstar R}{\star} \) to \( \sbfrac{R}{{\bf H}(R)}{\ast} \) sending \(o_\star(a, b)\) to \(o_\ast(u^{-2}(-b + v^2 + uv))\), where \(u \in R^\times\) and \(v \in R\) such that \(u = a - 2v\).
    \item \( \displaystyle \sbfrac{{\bf R}(R)}{R^\times \ltimes_\smallstar R}{\star} \) to \( \sbfrac{R}{{\bf T}(R)}{\pentagram} \) sending \(o_\star(a, b)\) to \(o_\ast(b + v^2)\), where \(v \in R\) such that \(a = 2v\). Moreover, \(\displaystyle R \times R = {\bf R}(R)\) when \(S\) is a \(\mathbb{Z}_{\langle p \rangle}\)-algebra, where \(p\) is a prime distinct from \(2\).
\end{itemize}
\end{proposition}

By examining the bijections at the end of the proposition, we observe that in these cases, only a single parameter is needed to describe such extensions up to isomorphism. Thus, the minimal number of parameters required is 1. This minimal number could be interpreted as the dimension of the moduli space.

\begin{remark} In this remark, we describe the moduli space at a given field $k$. 
\begin{enumerate}
\item {\sf Algebraically closed case: }
\begin{itemize} 
\item When $k$ is algebraically closed of characteristic not $2$, 
\begin{itemize} 
\item$\displaystyle \afz{{\sf F}}{2}{k}\simeq  \afz{{\sf R}}{2}{k}\simeq\afz{{\sf AS}}{2}{k} \simeq  \{ \operatorname{o}_\star (0,0) , \operatorname{o}_\star (0,1) \}$
\item $ \displaystyle \afz{{\sf SR}}{2}{k}\simeq \afz{{\sf SF}}{2}{k}\simeq\{\operatorname{o}_\star (0,1) \}$.
\end{itemize}
\item When $k$ is algebraically closed of characteristic $2$, 
\begin{itemize} 
\item$\displaystyle \afz{{\sf F}}{2}{k}\simeq  \{ \operatorname{o}_\star (0,0) , \operatorname{o}_\star (1,0) \}$ .
\item $\displaystyle \afz{{\sf SF}}{2}{k} \simeq \afz{{\sf AS}}{2}{k} \simeq \{  \operatorname{o}_\star (1,0) \}$.
\item $\displaystyle \afz{{\sf R}}{2}{k} \simeq \{ \operatorname{o}_\star (0,0) \}$.
\item $\displaystyle \afz{{\sf SR}}{2}{k}\simeq \emptyset$
\end{itemize}
\end{itemize} 
\item {\sf General case:} 
\begin{itemize} 
\item When $k$ is a field of characteristic not $2$, 
\begin{itemize} 
\item$\displaystyle \afz{{\sf F}}{2}{k} \simeq \afz{{\sf AS}}{2}{k}\simeq \afz{{\sf R}}{2}{k} \simeq  \{ \operatorname{o}_\star (0,a) | a \in k \} \simeq  \sbfrac{k}{k^\times}{\diamond}$.
\item $\displaystyle \afz{{\sf SF}}{2}{k}\simeq \afz{{\sf SR}}{2}{k} \simeq \{  \operatorname{o}_\star (0,a)| a\in k^\times \}\simeq  \sbfrac{k^\times}{k^\times}{\diamond}$.
\end{itemize}
\item When $k$ is a field of characteristic $2$, 
\begin{itemize} 
\item$\displaystyle \afz{{\sf F}}{2}{k}\simeq  \{ \operatorname{o}_\star (0,a) , \operatorname{o}_\star (1,b)|a \in k, b \in k \}\simeq \sbfrac{k}{k^\times}{\diamond} \coprod \sbfrac{k}{{\bf H}(k)}{\ast}$ .
\item $\displaystyle \afz{{\sf SF}}{2}{k} \simeq \afz{{\sf AS}}{2}{k} \simeq \{  \operatorname{o}_\star (1,b)|b \in k \}\simeq \sbfrac{k}{{\bf H}(k)}{\ast}$.
\item $\displaystyle \afz{{\sf R}}{2}{k} \simeq \{ \operatorname{o}_\star (0,a)|a\in k \}\simeq  \sbfrac{k}{k^\times}{\diamond}$.
\item $\displaystyle \afz{{\sf SR}}{2}{k}\simeq \emptyset$
\end{itemize}
\end{itemize} 
\end{enumerate}
\end{remark}

\subsection{The structure of the moduli spaces as a functor} Given an action \( \star \) of a group scheme \( G \) on a scheme \( X \) over \( S \), we denote by \( \sbfrac{X}{G}{\star} \) the functor that associates to any \( S \)-scheme \( T \) the quotient \( \sbfrac{X(R)}{G(R)}{\star} \) with respect to the induced action of the group \( G(R) \) on the set \( X(R) \). We will now introduce the affine schemes and group schemes necessary to define the various moduli spaces for free algebras of rank \(2\). The proof is omitted, as it can be easily obtained and is similar to the approach used in the proof of Proposition \ref{uni}.

\begin{deflem} \label{gs1}

\begin{enumerate}
\item We define the affine group scheme \(\mathbb{G}^S_m \ltimes_\smallstar \mathbb{G}^S_a\) to be the group scheme such that \(\mathbb{G}^S_m \ltimes_\smallstar \mathbb{G}^S_a (R) = R^\times \ltimes_\smallstar R\) for any \(S\)-algebra \(R\). This group scheme is representable by the Hopf algebra over \(S\), \(S[x, 1/x, y]\), endowed with:
\begin{itemize}
    \item the comultiplication that sends \(x\) to \(x \otimes x\) and \(y\) to \(x \otimes y + y \otimes 1\),
    \item the antipode that sends \(x\) to \(x^{-1}\) and \(y\) to \(-x^{-1} y\), and
    \item the counit that sends \(x\) to \(1\) and \(y\) to \(0\).
\end{itemize}

\item We define an action \(\star\) of \(\mathbb{G}^S_m \ltimes_\smallstar \mathbb{G}^S_a\) on \(\mathbb{A}_S^2\) via the coaction of the Hopf algebra \(S[x, 1/x, y]\) on the algebra \(S[z, t]\), sending \(z\) to \(1 \otimes 2x^{-1} y + z \otimes x^{-1}\) and \(t\) to \(1 \otimes x^{-2} y^2 + z \otimes x^{-2} y + t \otimes x^{-2}\).

\item We define the affine subscheme \(\mathbb{S}_S\) of \(\mathbb{A}_S^2\) as the scheme such that \(\mathbb{S}_S(R) = {\bf S}(R)\) for any \(S\)-algebra \(R\). \(\mathbb{S}_S\) is representable by the \(S\)-algebra \(\frac{S[x, 1/x, y, z]}{\langle y^2 - 4z - x \rangle}\) and also by \(\frac{S[x, y, z, t]}{\langle 2t - xz, -xt + 2yz - 1 \rangle}\). \(\mathbb{S}_S\) is a closed subscheme of \(\mathbb{A}_S^2 \times \mathbb{G}^S_m\) and also of $\mathbb{A}_S^4$. The action \(\star\) restricts to an action of \(\mathbb{G}^S_m \ltimes_\smallstar \mathbb{G}^S_a\) on \(\mathbb{S}_S\), which we continue to denote by \(\star\).

\item We define the sub-affine scheme \(\mathbb{R}_S\) of \(\mathbb{A}_S^2\) as the scheme such that \(\mathbb{R}_S(R) = {\bf R}(R)\) for any \(S\)-algebra \(R\). \(\mathbb{R}_S\) is representable by the \(S\)-algebra \(\frac{S[x, y, z]}{\langle x - 2z \rangle}\). \(\mathbb{R}_S\) is a closed subscheme of \(\mathbb{A}_S^3\). The action \(\star\) restricts to an action of \(\mathbb{G}^S_m \ltimes_\smallstar \mathbb{G}^S_a\) on \(\mathbb{R}_S\), which we continue to denote by \(\star\).

\item We define the sub-affine scheme \(\mathbb{AS}_S\) of \(\mathbb{A}_S^2\) as the scheme such that \(\mathbb{AS}_S(R) = {\bf AS}(R)\) for any \(S\)-algebra \(R\). \(\mathbb{AS}_S\) is representable by the \(S\)-algebra \(\frac{S[x, 1/x, y, z]}{\langle y + 2z + x \rangle}\). \(\mathbb{AS}_S\) is a closed subscheme of \(\mathbb{A}_S^2 \times \mathbb{G}^S_m\). The action \(\star\) restricts to an action of \(\mathbb{G}^S_m \ltimes_\smallstar \mathbb{G}^S_a\) on \(\mathbb{AS}_S\), which we continue to denote by \(\star\).

\item We define an action \(\diamond\) of the multiplicative group scheme \(\mathbb{G}^S_m\) over \(S\) on the affine line \(\mathbb{A}_S^1\) via the coaction of the Hopf algebra \(\mathbb{Z}[x, 1/x]\) on \(\mathbb{Z}[z]\), sending \(x\) to \(z \otimes x^2\). This action induces an action on \(\mathbb{G}^S_m\), which we also denote by \(\diamond\).

\item We define the closed affine subgroup scheme \(\mathbb{T}_S\) of \(\mathbb{G}^S_m \ltimes_\smallstar \mathbb{G}^S_a\) such that \(\mathbb{T}_S(R) = T(R)\) for any \(S\)-algebra \(R\). \(\mathbb{T}_S\) is represented by the Hopf algebra \(\frac{S[x, 1/x, y]}{\langle 2y \rangle}\). 

\item We define an action of \(\mathbb{T}_S\) on \(\mathbb{A}_S^1\) via the coaction of the Hopf algebra \(\frac{S[x, 1/x, y]}{\langle 2y + 1 - x\rangle }\) on \(S[z]\), sending \(z\) to \(1 \otimes 2\overline{x}^{-2}\overline{y}^2 + z \otimes \overline{x}^{-2}\).

\item We define the closed affine subgroup scheme \(\mathbb{H}_S\) of \(\mathbb{G}^S_m \ltimes_\smallstar \mathbb{G}^S_a\) as the group scheme such that \(\mathbb{H}_S(R) = {\bf H}(R)\) for any \(S\)-algebra \(R\). \(\mathbb{H}_S\) is representable by the Hopf algebra \(\frac{S[x, 1/x, y]}{\langle 2y + 1 - x \rangle}\), with:
\begin{itemize}
    \item the comultiplication sending \(\overline{x}\) to \(\overline{x} \otimes \overline{x}\) and \(\overline{y}\) to \(\overline{x} \otimes \overline{y} + \overline{y} \otimes \overline{1}\),
    \item the antipode sending \(\overline{x}\) to \(\overline{x}^{-1}\) and \(\overline{y}\) to \(-\overline{x}^{-1} \overline{y}\), and
    \item the counit sending \(\overline{x}\) to \(\overline{1}\) and \(\overline{y}\) to \(\overline{0}\).
\end{itemize}

\item We define the affine scheme \(\mathbb{L}_S\) such that \(\mathbb{L}_S(R) = L(R)\) for any \(S\)-algebra \(R\). \(\mathbb{L}_S\) is represented by the \(S\)-algebra \(\frac{S[x, 1/x, y]}{\langle 1 - 4y - x \rangle}\). \(\mathbb{L}_S\) is a closed subscheme of \(\mathbb{A}_S^1 \times \mathbb{G}_m^S\).

\item We define an action \(\ast\) of \(\mathbb{H}_S\) on \(\mathbb{A}_S^1\) via the coaction of the Hopf algebra \(\frac{S[x, 1/x, y]}{\langle 2y + 1 - x\rangle }\) on \(S[z]\) such that the comultiplication sends \(z\) to \(1 \otimes 2\overline{x}^{-1}\overline{y} + z \otimes \overline{x}^{-1}\). This action induces an action of \(\mathbb{H}_S\) on \(\mathbb{L}_S\), which we continue to denote by \(\ast\).

\end{enumerate}
\end{deflem}

%\begin{proof}
%We prove only statement $(4)$ as the other statements are proven similarly as in Proposition \ref{uni}.
%\begin{enumerate} 
%\item[(4)] To demonstrate the well-definedness of the coaction in $(4)$, it suffices to verify that 
%$$(2\overline{y}+1) \otimes (2\overline{y}+1)= 2 ((2\overline{y}+1)\otimes \overline{y} + \overline{y} \otimes 1 ) +1\otimes 1 .$$ 
%which is easily checked.
%\end{enumerate}
%\end{proof}

From Proposition \ref{propmain} and Definition-Lemma \ref{gs1}, we can deduce automatically the next theorem:
\begin{theorem}\label{main}
 The functor \(\af{{\sf P}}{2}{-}{S}\) is natually in correspondence with the functor

    \begin{itemize}
        \item \(\displaystyle \sbfrac{\mathbb{A}_S^2}{\mathbb{G}^S_m \ltimes_\smallstar \mathbb{G}_a^S}{\star} \), when \({\sf P} = {\sf F}\).
        \item \( \displaystyle \sbfrac{\mathbb{S}_S}{\mathbb{G}^S_m \ltimes_\smallstar \mathbb{G}_a^S}{\star} \), when \({\sf P} = {\sf SF}\).
        \item \( \displaystyle\sbfrac{\mathbb{R}_S}{\mathbb{G}^S_m \ltimes_\smallstar \mathbb{G}_a^S}{\star}\) and \( \sbfrac{\mathbb{A}_S^1}{\mathbb{T}_S}{\pentagram} \), when \({\sf P} = {\sf R}\).
        \item \(\displaystyle \sbfrac{\mathbb{AS}_S}{\mathbb{G}^S_m \ltimes_\smallstar \mathbb{G}_a^S}{\star} \) and \( \sbfrac{\mathbb{A}_S^1}{\mathbb{H}_S}{\ast} \), when \({\sf P} = {\sf AS}\).
        \item \(\displaystyle \sbfrac{\mathbb{A}_S^1}{\mathbb{G}^S_m}{\diamond} \), when \({\sf P} \in \{ {\sf F}, {\sf R}\}\) and \(S\) is a \(\mathbb{Z}_{\langle p \rangle}\)-algebra, where \(p\) is a prime distinct from \(2\).
        \item \(\displaystyle \sbfrac{\mathbb{G}^S_m}{\mathbb{G}^S_ms}{\diamond} \), when \({\sf P} = {\sf SR}\) and \(S\) is a \(\mathbb{Z}_{\langle p \rangle}\)-algebra, where \(p\) is a prime distinct from \(2\).
        \item \( \displaystyle \sbfrac{\mathbb{L}_S}{\mathbb{H}_S}{\ast} \), when \({\sf P} = {\sf SAS}\).
    \end{itemize}
We have the following embeddings 
\begin{itemize} 
\item $\displaystyle \sbfrac{\mathbb{S}_S}{\mathbb{G}^S_m \ltimes_\smallstar \mathbb{G}_a^S}{\star}\hookrightarrow \sbfrac{\mathbb{A}_S^2}{\mathbb{G}^S_m \ltimes_\smallstar \mathbb{G}_a^S}{\star}$, 
\item \(\displaystyle \sbfrac{\mathbb{A}_S^1}{\mathbb{T}_S}{\pentagram} \simeq  \sbfrac{\mathbb{R}_S}{\mathbb{G}^S_m \ltimes_\smallstar \mathbb{G}_a^S}{\star}  \hookrightarrow \sbfrac{\mathbb{A}_S^2}{\mathbb{G}^S_m \ltimes_\smallstar \mathbb{G}_a^S}{\star} \),
\item \( \displaystyle \sbfrac{\mathbb{L}_S}{\mathbb{H}_S}{\ast} \hookrightarrow \sbfrac{\mathbb{A}_S^1}{\mathbb{H}_S}{\ast} \simeq \sbfrac{\mathbb{AS}_S}{\mathbb{G}^S_m \ltimes_\smallstar \mathbb{G}_a^S}{\star}\hookrightarrow  \sbfrac{\mathbb{A}_S^2}{\mathbb{G}^S_m \ltimes_\smallstar \mathbb{G}_a^S}{\star}   \).
\end{itemize}
\end{theorem}

%\begin{theorem}\label{main}
%There is a natural correspondence between
%\begin{enumerate} 
%\item $\af{F}{n}{-}{S}$ and $\frac{ \mathbb{A}_S^2}{(\mathbb{G}^S_m \ltimes_\smallstar \mathbb{G}_a^S)^\star}$.
%\item $\af{P}{n}{-}{S_{(p)}}$ and $\frac{ \mathbb{A}_{S_{(p)}}^1}{(\mathbb{G}^{S_{(p)}}_m)^\diamond}$, where $p$ prime different $2$.
%\item $\af{P}{n}{-}{S_{(2)}}$ and $\frac{ \mathbb{A}_{S_{(2)}}^1}{(\mathbb{H}_{S_{(2)}})^\ast}$.
%\item $\afz{AS}{2}{-}$ and $\frac{ \mathbb{A}_{S}^1}{(\mathbb{H}_{S})^\ast}$.
%\item $\afz{SF}{2}{-}$ and  $\frac{\mathbb{L}_S}{(\mathbb{H}_S)^\ast}$.
%\item  $\afz{SR}{2}{-}$ and $\frac{ \mathbb{G}^{S}_m}{(\mathbb{G}^{S}_m)^\diamond}$.
%\item $\afz{R}{2}{-}$ and $\frac{ \mathbb{A}_S^1}{(\mathbb{T}_S)^\pentagram}$.
%\end{enumerate}
%\end{theorem}

\bibliographystyle{Abbrv}

\begin{thebibliography}{1}

%\bibitem{bourbaki}
%N.~Bourbaki,
%\newblock Alg\`ebre commutative.  Chapitres 1 \`a 4.
%\newblock {\em Springer-Verlag Berlin}, 2006. \href{https://doi.org/10.1007/978-3-540-33976-2}{Doi}

\bibitem{bhargava} M. Bhargava, Higher composition laws III: The parametrization of quartic rings. {\em Annals of Mathematics}, 159, 1329-1360, 2004.   

\bibitem{bhargava2}  M. Bhargava, Higher composition laws IV: The parametrization of quintic rings. Annals of Mathematics, 167, 5394, 2008.

\bibitem{mpendulo}
M.~Cele, and S.~Marques,
\newblock The geometry of the moduli space of non-cyclic biquadratic field extensions.
\newblock {\em Quaestiones Mathematicae} (45)10, 2022.

\bibitem{Dallaporta}
W. Dallaporta,
Recovering the Picard group of quadratic algebras from Wood's binary quadratic forms. 
Int. J. Number Theory 21, No. 4, 739-767 (2025).

\bibitem{Delone} B. N. Delone, and D. K. Faddeev, The Theory of Irrationalities of the Third Degree.
{\em Translations of Mathematical Monographs}
Volume: 10; 509 pp, 1964.

\bibitem{Meyer}
F., Demeyer, and E. Ingraham
\newblock Separable Algebras over Commutative Rings. 
\newblock {\em Lecture Notes Math.}, Vol. 18 I.,1971. Berlin-Heidelberg-New York: Springer.
 
\bibitem{grothendieck}
A.~Grothendieck; J.~Dieudonn\'e,
\newblock \'El\'ements de g\'eom\'etrie alg\'ebrique IV: \'Etude locale des sch\`emas et des morphismes de sch\'emas. Deuxi\`eme partie.
\newblock {\em Publications Math\'ematiques de l'IH\`ES.} 20: 5-259. doi:10.1007/bf02684747. MR 0173675

\bibitem{hahn} 
A. J. Hahn. 
\newblock Quadratic Algebras, Clifford Algebras, and Arithmetic Witt Groups. 
\newblock {\em Springer-Verlag.} 1994.  Universitext.

\bibitem{Kitamura}
K.~Kitamura,
\newblock On the free quadratic extensions of a commutative ring.
\newblock {\em Osaka Journal of Mathematics.} 10(1): 15-20 (1973). doi:10.18910/3522


%\bibitem{liu}
%Q.~Liu,
%\newblock Algebraic Geometry and Arithmetic Curves 
%\newblock {\em Oxford Graduate Texts in Mathematics}, (2002). ISBN-13: 978-0199202492

\bibitem{marques}
S.~Marques; K.~Ward,
Cubic fields: a primer. {\em European Journal of Mathematics} 5, 551-570 (2019). https://doi.org/10.1007/s40879-018-0258-5

\bibitem{Meyer}
F., Demeyer, and E. Ingraham
\newblock Separable Algebras over Commutative Rings. 
\newblock {\em Lecture Notes Math.}, Vol. 18 I.,1971. Berlin-Heidelberg-New York: Springer.

\bibitem{micali}
A.~Micali, and Ph. Revoy,
\newblock Modules quadratiques.
\newblock {\em Bull. Soc. Math. M\'emoire 63} (1979).

\bibitem{Pirashvili}
I. Pirashvili
\newblock On the group of separable quadratic algebras and stacks 
https://arxiv.org/abs/1604.02547


\bibitem{poonen}
B.~Poonen, The moduli space of commutative algebras of finite rank. {\em J. Eur. Math. Soc.} 10 , no. 3, pp. 817-836 (2008). doi:10.4171/JEMS/131

\bibitem{small} 
C. Small, 
The group of quadratic extensions. 
Journal of Pure and Applied Algebra
2, no 2, (1972), pp 83-105



\bibitem{Szeto}
G.~Szeto, Y. F.~Wong,
\newblock {On free quadratic extensions of rings.}
\newblock {\em Monatsh. Math.}, (92)4: 323--328, (1981). doi:10.1007/BF01320063

\bibitem{voight}
J.~Voight,
Discriminants and the monoid of quadratic rings. {\em Preprint} 
https://doi.org/10.2140/pjm.2016.283.483, (2015).
 





%\bibitem{bconrad}
%B., Conrad, 
%\newblock Existence of the discriminant ideal. 
%\newblock \href{https://math.stanford.edu/~conrad/676Page/handouts/discexist.pdf}{Link}

\end{thebibliography}

\end{document}